\documentclass[fleqn,11pt]{article}
\usepackage{tikz}
\usetikzlibrary{arrows,shapes,chains}
\usepackage{graphics}
\usepackage{color}
\usepackage{mathrsfs}
\usepackage{amssymb}
\usepackage{amsmath}
\usepackage{multirow}
\usepackage{booktabs}
\begin{document}

\pagenumbering{arabic}
\newcounter{comp1}

\newtheorem{definition}{Definition}[section]
\newtheorem{proposition}{Proposition}[section]
\newtheorem{example}{Example}[section]
\newtheorem{method}{Method}[section]
\newtheorem{lemma}{Lemma}[section]
\newtheorem{theorem}{Theorem}[section]
\newtheorem{corollary}{Corollary}[section]
\newtheorem{application}{Application}[section]
\newtheorem{assumption}{Assumption}
\newtheorem{algorithm}{Algorithm}[section]
\newtheorem{remark}{Remark}[section]
\newcommand{\fig}[1]{\begin{figure}[hbt]
                  \vspace{1cm}
                  \begin{center}
                  \begin{picture}(15,10)(0,0)
                  \put(0,0){\line(1,0){15}}
                  \put(0,0){\line(0,1){10}}
                  \put(15,0){\line(0,1){10}}
                  \put(0,10){\line(1,0){15}}
                  \end{picture}
                  \end{center}
                  \vspace{.3cm}
                  \caption{#1}
                  \vspace{.5cm}
                  \end{figure}}
\newcommand{\Axk}{A(x^k)}
\newcommand{\Aumb}{\sum_{i=1}^{N}A_{i}u_{i}-b}
\newcommand{\Kk}{K^k}
\newcommand{\Kki}{K_{i}^{k}}
\newcommand{\Aukmb}{\sum_{i=1}^{N}A_{i}u_{i}^{k}-b}
\newcommand{\Au}{\sum_{i=1}^{N}A_{i}u_{i}}
\newcommand{\Aukpmb}{\sum_{i=1}^{N}A_{i}u_{i}^{k+1}-b}
\newcommand{\nab}{\nabla^2 f(x^k)}
\newcommand{\xk}{x^k}
\newcommand{\ubk}{\overline{u}^k}
\newcommand{\uhk}{\hat u^k}
\def\QEDclosed{\mbox{\rule[0pt]{1.3ex}{1.3ex}}} % ����ʵ�ķ�
\def\QEDopen{{\setlength{\fboxsep}{0pt}\setlength{\fboxrule}{0.2pt}\fbox{\rule[0pt]{0pt}{1.3ex}\rule[0pt]{1.3ex}{0pt}}}}
% �������ķ�
\def\QED{\QEDopen} % ѡ��\QEDclosed �õ�ʵ��
\def\proof{\par\noindent{\em Proof.}}
\def\endproof{\hfill $\Box$ \vskip 0.4cm}
\newcommand{\RR}{\mathbf R}

\title {\bf
Linear Convergence of Variable Bregman Stochastic Coordinate Descent Method for Nonsmooth Nonconvex Optimization by Level-set Variational Analysis}
\author{Lei Zhao\thanks{Antai College of Economics and Management and Sino-US Global Logistics
Institute, Shanghai Jiao Tong University, Shanghai, China({\tt l.zhao@sjtu.edu.cn})}, Daoli Zhu\thanks {Antai College of Economics and Management and Sino-US Global Logistics
Institute, Shanghai Jiao Tong University, Shanghai, China({\tt
dlzhu@sjtu.edu.cn})}}
\footnotetext[1]{Acknowledgments: this research was supported by NSFC:71471112 and NSFC:71871140}

\maketitle

\begin{abstract}
Large-scale nonconvex and nonsmooth problems have attracted considerable attention in the fields of compress sensing, big data optimization and machine learning. Exploring effective methods is still the main challenge of today's research. Stochastic coordinate descent type methods have been widely used to solve large-scale optimization problems. In this paper, we derive the convergence of variable Bregman stochastic coordinate descent (VBSCD) method for a broad class of nonsmooth and nonconvex optimization problems, i.e., any accumulation of the sequence generated by VBSCD is almost surely a critical point. Moreover, we develop a new variational approach on level sets that aim towards the convergence rate analysis. If the level-set subdifferential error bound holds, we derive a linear rate of convergence for the expected values of the objective function and expected values of random variables generated by VBSCD.

\vspace{1cm}

%\vspace{0.5cm}
%
%\noindent {\bf Mathematics Subject Classification:}

\end{abstract}
%La m\'{e}thode du probl\`{e}me auxiliaire permet d'obtenir la solution
%de probl\`{e}mes de minimisation par la r\'{e}solution d'une suite
%de probl\`{e}mes plus simples.
%Nous g\'{e}n\'{e}ralisons cette m\'{e}thode en la couplant
%avec une recherche unidimensionnelle (``linesearch'').
%Nous proposons un algorithme fondamental et \'{e}tudions
%ses propri\'{e}t\'{e}s de convergence
%pour la minimisation de fonctionnelles pseudoconvexes
%en dimension infinie. Nous introduisons \'{e}galement la m\'{e}thode
%du probl\`{e}me auxiliaire partiel de descente qui ne
%lin\'{e}arise que l'objectif et n'introduit qu'un
%sous-ensemble des variables dans le terme auxiliaire.
\normalsize
%\newpage
\vspace{1cm}
\section{Introduction}
\indent This paper considers the following nonconvex and nonsmooth optimization problem:
\begin{equation}\label{Prob:general-function}
\mbox{{\rm(P)}}\qquad\min_{x\in\RR^n}\qquad F(x)=f(x)+g(x)=f(x)+\sum_{i=1}^{N}g_i(x_i)
\end{equation}
where $f:\RR^n\rightarrow(-\infty,\infty]$ is a $\mathcal{C}^1$-smooth function (may be nonconvex) and $g:~\RR^n\rightarrow(-\infty, \infty]$ is a continuous semi-convex function. Moreover, $g_i:~\RR^{n_i}\rightarrow(-\infty, \infty]$, $x_i\in\RR^{n_i}$ and $\sum\limits_{i=1}\limits^{N}n_i=n$.
\subsection{Notations and assumptions}
\indent Throughout this paper, $\langle\cdot,\cdot\rangle$ and $\|\cdot\|$ denote the Euclidean scalar product of $\RR^n$ and its corresponding norm respectively.\\
\indent Let $\mathbf{C}$ be a subset of $\RR^n$ and $x$ be any point in $\RR^n$. Define $dist(x,\mathbf{C})=\inf\{\|x-z\|:z\in\mathbf{C}\}$. When $\mathbf{C}=\emptyset$, we set $dist(x,\mathbf{C})=\infty$.\\
\indent Moreover, we use $\mathcal{F}(\RR^n)$ to denote the set of $\mathcal{C}^1$-smooth functions from $\RR^n$ to $(-\infty,+\infty]$; $\Gamma(\RR^n)$ is the set of continuous functions from $\RR^n$ to $(-\infty,+\infty]$.\\
\indent Given an $\overline{x}\in \mathbf{dom}~F$, let $\overline{F}=F(\overline x)$, set $[F\leq\overline{F}]=\{x\in\RR^n|F(x)\leq F(\overline{x})\}$ and $[F>\overline{F}]=\{x\in\RR^n|F(x)>F(\overline{x})\}$.\\
\indent Additionally, the subdifferential calculus $\partial$ that we will use throughout the paper is limiting-subdifferential, which is standard in variational analysis~\cite{Rockafellar2009}.\\
\indent Throughout the remainder of this paper, we make the following assumption on $f$ and $g$.
%%%%ASSUMPTION_ONE%%%%%%%%
%%%%%%%%%%%%%%%%%%%%%%
\begin{assumption}\label{assump1}
\begin{itemize}
\item[{\rm(i)}] $f\in\mathcal{F}(\RR^n)$ is a nonconvex differentiable function with $\mathbf{dom}~f$ convex. Its gradient $\nabla f$ is $L-$Lipschitz continuous on $\mathbf{dom}~f$ such that
\begin{equation*}
\frac{L}{2}\|y-x\|^2+\langle\nabla f(x),y-x\rangle\geq f(y)-f(x)~~\hspace{4mm}\forall x,y\in \mathbf{dom}~f.
\end{equation*}
\item[{\rm(ii)}] $g\in\Gamma(\RR^n)$ and $\mathbf{dom}~g$ is a convex set. Moreover, $g$ is semi-convex on $\mathbf{dom}~g$ with modulus $\rho>0$, one has
\begin{equation}\label{eq:semi-convex}
g(y)\geq g(x)+\langle\xi,y-x\rangle-\frac{\rho}{2}\|x-y\|^2, \forall\xi\in\partial g(x),~x,y\in \mathbf{dom}~g.
\end{equation}
\item[{\rm(iii)}] $F$ is level-bounded i.e., the set $\{x\in\RR^n|F(x)\leq r\}$ is bounded (possibly empty) for every $r\in\RR$.
%\item[{\rm(iv)}]
\end{itemize}
\end{assumption}
\indent From {\rm{(i)} and {\rm{(ii)}}, $\mathbf{dom}~F$ is a convex set. In addition,
as a consequence of {\rm{(iii)},  the optimal value $F^*$ of (P) is finite and the optimal solution set $\mathbf{X}^*$ of (P) is non-empty. Moreover, the set of all critical points of $F$ is denoted by $\overline{\mathbf{X}}=\{x|0\in\nabla f(x)+\partial g(x)\}$.\\
\indent Here we note that the well known SCAD penalty~\cite{SCAD} and MCP penalty~\cite{MCP} are both semi-convex.\\
\indent To construct an algorithm for (P), we introduce Bregman distance function $D$ and parameter $\epsilon$.\\
\indent Let $K:\RR^n\rightarrow(-\infty,\infty)$ be a twice differentiable strongly convex function. The Bregman distance function associated with $K$ is defined by
$$D(x,y)=K(y)-[K(x)+\langle\nabla K(x),y-x\rangle].$$
Then we have $\nabla_yD(x,y)=\nabla K(y)-\nabla K(x)$, $\nabla_xD(x,y)=\langle\nabla^2K(x),x-y\rangle$ and $\nabla_yD(x,x)=0$. We make the following standing assumption on function $K$ and parameter $\epsilon$.
\begin{assumption}\label{assump2}
\begin{itemize}
\item[{\rm(i)}] $K$ is strongly convex with $m$ and with its gradient $\nabla K$ being $M$-Lipschitz.
\item[{\rm(ii)}] The parameter $\epsilon$ satisfies: $0<\underline{\epsilon}\leq\epsilon\leq\overline{\epsilon}<\min\{\frac{m}{L},\frac{m}{\rho}\}$.
\end{itemize}
\end{assumption}
\indent Under this assumption, we have Bregman distance function $D$ satisfies:
\begin{eqnarray}
\frac{m}{2}\|x-y\|^2\leq D(x,y)\leq\frac{M}{2}\|x-y\|^2.
\end{eqnarray}
\indent Now we are ready to introduce the variable Bregman stochastic coordinate descent method.
\subsection{Variable Bregman Stochastic Coordinate Descent method}\label{sec:Variable Bregman}
\indent In this subsection we introduce the Variable Bregman Stochastic Coordinate Descent (VBSCD) method. First, recall that a variable Bregman distance-like function has the form
\begin{equation}
D^k(x,y)=K^k(y)-[K^k(x)+\langle\nabla K^k(x),y-x\rangle].
\end{equation}
\indent We propose to solve (P) by generating a sequence $\{x^k\}$ using the following VBSCD method:\\
\noindent\rule[0.25\baselineskip]{\textwidth}{1.5pt}
{\bf Variable Bregman Stochastic Coordinate Descent method (VBSCD)}\\
\noindent\rule[0.25\baselineskip]{\textwidth}{0.5pt}
{Initialize} $x^0\in\RR^n$\\
 \textbf{for} $k = 0,1,\cdots $, \textbf{do}
\begin{eqnarray}\label{APk}
&&\mbox{Choose $i(k)$ from $\{1,2,...,N\}$ with equal probability}\nonumber\\
&&\mbox{(AP$^{i(k)}$)}\; x^{k+1}\in\arg\min_{x\in\RR^n}\bigg{\{}\langle\nabla_{i(k)} f(x^k),(x-x^k)_{i(k)}\rangle+g_{i(k)}(x_{i(k)})+\frac{1}{\epsilon^k}D^k(x^k,x)\bigg{\}}.\nonumber\\
\end{eqnarray}
\textbf{end for}\\
\noindent\rule[0.25\baselineskip]{\textwidth}{1.5pt}
Here we note that if $K^k(\cdot)=\frac{1}{2}\|\cdot\|^2$, VBSCD refines the classical random coordinate descent scheme. (see~\cite{LuXiao2015,Nesterov2012,Necoara2015,RT2014})
\subsection{Error bounds and their relationship}\label{subsec:eb}
\indent In this subsection, let $\overline{x}\in\overline{\mathbf{X}}$, we introduce four types of error bounds which are always used in the convergence rate analysis of algorithms.\\
\indent For given positive numbers $\eta$ and $\nu$, let $\mathbb{B}(\overline{x};\eta)=\{x\in\RR^n~|~\|x-\overline{x}\|\leq\eta\}$ and
$$\mathfrak{B}(\overline{x};\eta,\nu)=\mathbb{B}(\overline{x};\eta)\cap \{x\in \RR^n~|~\overline{F}<F(x)<\overline{F}+\nu\},$$
\indent Then we will introduce the concepts of the level-set subdifferential error bound (LS-EB).
\begin{definition}[Level-set subdifferential error bound (LS-EB)]\label{defi:LSEB}
The proper lower semicontinuous function $F$  is said to satisfy the level-set subdifferential error bound condition at $\overline{x}$ if there exist $\eta>0$, $\nu>0$, and $c_0>0$ such that the following inequality holds:
$$dist\left(x,[F\leq\overline{F}]\right)\leq c_0dist\left(0,\partial F(x)\right)~~\forall
x\in\mathfrak{B}(\overline{x};\eta,\nu).$$
\end{definition}
\indent For a given Bregman Proximal Mapping, we can introduce Bregman proximal error bound. The definition of Bregman Proximal Mapping is as follows.\\
{\bf Bregman Proximal Mapping}\\
Bregman proximal mapping $T_{D,\epsilon}$ is defined by
\begin{equation}\label{defi:Tk}
T_{D,\epsilon}(x)=\arg\min_{y\in\RR^n}\langle\nabla f(x),y-x\rangle+g(y)+\frac{1}{\epsilon}D(x,y),\quad\forall x\in\RR^n.
\end{equation}
\indent Bregman proximal error bound (BP-EB) is defined as follows.
\begin{definition}[Bregman proximal error bound (BP-EB)]\label{defi:BPEB}
Given a Bregman function $D$ along with $\epsilon>0$, we say that the Bregman proximal error bound holds at $\overline{x}\in\overline{\mathbf{X}}$, if there exist $\eta>0$, $\nu>0$ and $c_1>0$ such that
$$dist\left(x,\overline{\mathbf{X}}\right)\leq c_1dist\left(x,T_{D,\epsilon}(x)\right)~~\forall
x\in\mathfrak{B}(\overline{x};\eta,\nu).$$
\end{definition}
\indent The well known Kurdyka-{\L}ojasiewicz is defined as follows. (see~\cite{Attouch13, KL19})
\begin{definition}[Kurdyka-{\L}ojasiewicz property (KL)]\label{defi:KL}
The proper lower semi continuous function $F$ is said to satisfy the Kurdyka-{\L}ojasiewicz (KL) property at $\overline{x}$, if there exist $\eta>0$, $\nu>0$ and $c_2>0$ such that the following inequality holds:
$$dist\left(0,\partial F(x)\right)\geq c_2[F(x)-\overline{F}]^{\frac{1}{2}}~~\forall
x\in\mathfrak{B}(\overline{x};\eta,\nu).$$
\end{definition}
\indent By introducing $Prox_g^{\epsilon}\left(x-\epsilon\nabla f(x)\right)=\arg\min\limits_{y\in\RR^n}\langle\nabla f(x),y-x\rangle+g(y)+\frac{1}{2\epsilon}\|x-y\|^2$, we have Luo and Tseng's error bound (LT-EB) as follows. (see~\cite{LuoTseng92,Necoara2015,Tseng2009})
\begin{definition}[Luo and Tseng's error bound (LT-EB)]\label{defi:LPEB}
We say the Luo-Tseng error bound holds if any $\varepsilon_1\geq\inf_{x\in\RR^n} F(x)$, there exists constant $c_3>0$ and $\varepsilon_2>0$ such that
\begin{equation}
dist(x,\overline{\mathbf{X}})\leq c_3\|x-Prox_g^{\epsilon}\left(x-\epsilon\nabla f(x)\right)\|
\end{equation}
whenever $F(x)\leq\varepsilon_1$, $\|x-Prox_g^{\epsilon}\left(x-\epsilon\nabla f(x)\right)\|\leq\varepsilon_2$.
\end{definition}
Additionally, we introduce two standard assumptions, which are always used together with above error bounds.
\begin{assumption}[~\cite{LuoTseng92,Necoara2015,Tseng2009}]\label{assumpH}
There exists $\rho>0$ such that $\|x-y\|\geq\rho$ whenever $x,y\in\overline{\mathbf{X}}$ with $F(x)\neq F(y)$.
\end{assumption}
\begin{assumption}[Growth condition {~\cite{Attouch13}}]\label{assump3}
For any $\rho>0$ there exist $0<\sigma<\rho$, $\delta>0$ and $\mathfrak{a}>0$ such that
\begin{eqnarray}
\begin{array}{c}
x\in\mathbb{B}(\overline{x},\sigma), F(x)<\overline{F}+\delta\\
y\notin\mathbb{B}(\overline{x},\rho)
\end{array}
\bigg{\}}\Rightarrow F(x)<F(y)+\mathfrak{a}\|y-x\|^2.
\end{eqnarray}
\end{assumption}
\indent Here we note that, according to Zhu and Deng~\cite{ZhuDeng2019}, the combination of BP-EB, semi-convex $g$ and Assumption~\ref{assumpH} implies LS-EB. If $g$ is convex, LT-EB implies BP-EB. Additionally, if $g$ is semi-convex, KL and LS-EB are equivalent. Moreover, we observe that LT-EB is a global version error bound (EB). BP-EB, KL and LS-EB are local-version error bounds. Therefore, LS-EB is the weakest EB-type condition. (see Figure~\ref{fig:1})
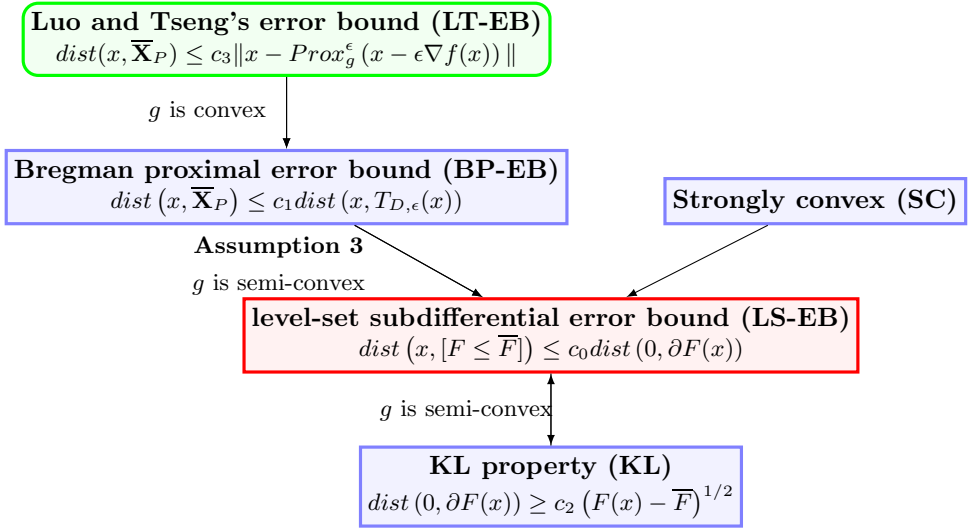
\begin{figure}
\centering
\begin{center}
\scriptsize
		\tikzstyle{format}=[rectangle,draw,thin,fill=white]
		\tikzstyle{test}=[diamond,aspect=2,draw,thin]
		\tikzstyle{point}=[coordinate,on grid,]
\begin{tikzpicture}
[%��ͷ��ģʽ��latex Ĭ��ģʽ
>=latex,
%�����������
node distance=5mm,
% hv path ��ʾһ�����㵽��һ����������ˮƽ�ٴ�ֱ��vh �෴��skip loop ��ʾ
%��ֱ-ˮƽ-��ֱ vskip loop ��ʾˮƽ-��ֱ-ˮƽ
 ract/.style={draw=blue!50, fill=blue!5,rectangle,minimum size=6mm, very thick, font=\itshape, align=center},
 racc/.style={rectangle, align=center},
 ractm/.style={draw=red!100, fill=red!5,rectangle,minimum size=6mm, very thick, font=\itshape, align=center},
 cirl/.style={draw, fill=yellow!20,circle,   minimum size=6mm, very thick, font=\itshape, align=center},
 raco/.style={draw=green!500, fill=green!5,rectangle,rounded corners=2mm,  minimum size=6mm, very thick, font=\itshape, align=center},
 hv path/.style={to path={-| (\tikztotarget)}},
 vh path/.style={to path={|- (\tikztotarget)}},
 skip loop/.style={to path={-- ++(0,#1) -| (\tikztotarget)}},
 vskip loop/.style={to path={-- ++(#1,0) |- (\tikztotarget)}}]
        \node (a) [ractm]{\baselineskip=3pt\small {\bf level-set subdifferential error bound (LS-EB)}\\ \baselineskip=3pt\footnotesize$dist\left(x,[F\leq\overline{F}]\right)\leq c_0{dist}\left(0,\partial F(x)\right)$};
        \node (b1) [ract, above = of a, xshift=-100, yshift=13]{\baselineskip=3pt\small {\bf Bregman proximal error bound (BP-EB)}\\
        \baselineskip=3pt\footnotesize$dist\left(x,\overline{\mathbf{X}}_P\right)\leq c_1dist\left(x,T_{D,\epsilon}(x)\right)$};
        \node (b3) [ract, above = of a, xshift=100, yshift=13]{\baselineskip=3pt\small {\bf Strongly convex (SC)}};
        \node (bb) [racc, below= of b1, xshift=-3, yshift=13]{\baselineskip=3pt\footnotesize{\bf Assumption~\ref{assumpH}}};
        \node (bbb) [racc, below= of bb, yshift=14]{\baselineskip=3pt\footnotesize $g$ is semi-convex};
        \node (b2) [ract, below = of a, yshift=-13]{\baselineskip=3pt\small {\bf KL property (KL)}\\ \baselineskip=3pt\footnotesize $dist\left(0,\partial F(x)\right)\geq c_2\left(F(x)-\overline{F}\right)^{1/2}$};
        \node (bbbc) [racc, above= of b2, xshift=-32, yshift=-8]{\baselineskip=3pt\footnotesize $g$ is semi-convex};
        \node (c) [raco, above = of b1, yshift=13]{\baselineskip=3pt\small {\bf Luo and Tseng's error bound (LT-EB)}\\
        \baselineskip=3pt\footnotesize$dist(x,\overline{\mathbf{X}}_P)\leq c_3\|x-Prox_g^{\epsilon}\left(x-\epsilon\nabla f(x)\right)\|$};
        \node (bbc) [racc, below= of c, xshift=-30, yshift=8]{\baselineskip=3pt\footnotesize$g$ is convex};
        \path %(a) edge[->] (b)
               (c) edge[->] (b1)
               (c) edge[->] (b1)
              (b1) edge[->] (a)
              (b1) edge[->] (a)
              (b2) edge[->] (a)
              (b3) edge[->] (a)
               (a) edge[->] (b2);
\end{tikzpicture}
\caption{The relationship among the notions of the LS-EB, BP-EB, KL, LT-EB and Strongly convex (SC)}\label{fig:1}
\end{center}
\end{figure}
\subsection{Related works}
\indent In this subsection, we compare the VBSCD method and existing theoretical results of the stochastic coordinate descent-type methods on linear convergence. (see Table~\ref{table1}) The difference of this paper compared to existing research is our analysis of the convergence and convergence rate using a local version error bound condition (level-set subdifferential error bound). In other words, the error bound condition holds in the neighborhood of a given point. We analyze three cases in this paper: the critical point, local minimum and global minimum.
\begin{table}[htbp]
\centering  % 显示位置为中间
\caption{Existence results of Variable Bregman Stochastic Coordinate Descent method (VBSCD)}  % 表格标题
\label{table1}
\begin{tabular}{|c|c|c|}
\hline
\multirow{2}{2cm}{{\bf Paper}}&{\bf Problem property}&\multirow{2}{2cm}{{\bf Theoretical Results}}\\
\cmidrule{2-2}
&{\bf Algorithm}&\\
\hline  % 表格的横线
\multirow{2}{2cm}{Nesterov, 2012~\cite{Nesterov2012}}&\begin{tabular}{c}$f\in\mathcal{F}(\RR^n)$ convex;\\$g=0$;\\$F$ is strongly convex.\end{tabular}&\multirow{2}{2cm}{$\{\mathbb{E}_{\xi_{k-1}}F(x^k)\}$ $Q-$linear}\\
\cmidrule{2-2}
&$K^k(\cdot)=\frac{1}{2}\|\cdot\|^2$&\\
\hline
\multirow{2}{2cm}{Lu \&Xiao, 2015~\cite{LuXiao2015}}&\begin{tabular}{c}$f\in\mathcal{F}(\RR^n)$ convex;\\$g\in\Gamma(\RR^n)$ convex;\\$F$ is strongly convex.\end{tabular}&\multirow{2}{2cm}{$\{\mathbb{E}_{\xi_{k-1}}F(x^k)\}$ $Q-$linear}\\
\cmidrule{2-2}
&$K^k(\cdot)=\frac{1}{2}\|\cdot\|^2$&\\
\hline
\multirow{2}{2cm}{Patrascu \&Necoara, 2015~\cite{Necoara2015}}&\begin{tabular}{c}$f\in\mathcal{F}(\RR^n)$;\\$g\in\Gamma(\RR^n)$ convex;\\LT-EB+Assumption~\ref{assumpH}.\end{tabular}&\multirow{2}{2cm}{$\{\mathbb{E}_{\xi_{k-1}}F(x^k)\}$ $Q-$linear}\\
\cmidrule{2-2}
&$K^k(\cdot)=\frac{1}{2}\|\cdot\|^2$&\\
\hline
\multirow{4}{2cm}{This paper}&\begin{tabular}{c}$f\in\mathcal{F}(\RR^n)$;\\ $g\in\Gamma(\RR^n)$ semi-convex;\end{tabular}&\begin{tabular}{l}{\bf $\overline{x}$ is a critical point}\\+Assumption~\ref{assumpH}\\
$\{\mathbb{E}_{\xi_{k-1}}F(x^k)\}$ $Q-$linear;\\ $\{\mathbb{E}_{\xi_{k-1}}x^k\}$ $R-$linear\end{tabular}\\
\cmidrule{3-3}
&LS-EB on $\mathfrak{B}(\overline{x};\eta,\nu)$.&\begin{tabular}{l}{\bf $\overline{x}$ is a local minimum}\\+Assumption~\ref{assump3}\end{tabular}\\
\cmidrule{2-2}
&\multirow{2}{2cm}{$K^k(\cdot)$ satisfy Assumption~\ref{assump2}.}&\begin{tabular}{l}$\{\mathbb{E}_{\xi_{k-1}}F(x^k)\}$ $Q-$linear;\\ $\{\mathbb{E}_{\xi_{k-1}}x^k\}$ $R-$linear\end{tabular}\\
\cmidrule{3-3}
&&\begin{tabular}{l}{\bf $\overline{x}$ is a global minimum}\\$\{\mathbb{E}_{\xi_{k-1}}F(x^k)\}$ $Q-$linear;\\ $\{\mathbb{E}_{\xi_{k-1}}x^k\}$ $R-$linear\end{tabular}\\
\hline
\end{tabular}
\end{table}
\subsection{Main contributions and outline of this paper}
In this paper, we propose a variable Bregman stochastic coordinate descent (VBSCD) method based on the Variable Bregman Proximal Gradient (VBPG) method (Zhu and Deng, 2019~\cite{ZhuDeng2019}, Cohen 1980~\cite{Cohen80}) for (P). In this method, we randomly update a block of variables based on the uniform distribution in each iteration. The sequence generated by our algorithm is proven to converge to a critical point of problem (P). Moreover, we develop a new variational approach on level sets that aim towards the convergence rate analysis. The $\{\mathbb{E}_{\xi_{k-1}}F(x^k)\}$ $Q-$linear convergence rate and $\{\mathbb{E}_{\xi_{k-1}}x^k\}$ $R-$linear convergence rate are analyzed in this paper.\\
The remainder of this paper is organized as follows. Section~\ref{sec:pre} is devoted to the properties of the Bregman-type mappings and functions. In Section~\ref{sec:convergence} we establish the convergence of VBSCD. In Section~\ref{sec:convergence rate}, the linear convergence rates of three cases are analyzed.

\section{Basic properties of Bregman type mappings and functions}\label{sec:pre}
The analysis of convergence and rate of convergence for the VBSCD method, essentially relies on Bregman type mappings and functions. Given a Bregman function $D$ and a positive $\epsilon$, the following mappings and functions will play a key role for the analysis of convergence and rate of convergence for the VBSCD method.\\
{\bf Bregman Proximal Envelope Function (BP Envelope Function)}\\
BP envelope function $E_{D,\epsilon}$ is defined by
\begin{equation}
E_{D,\epsilon}(x)=\min_{y\in\RR^n}\{f(x)+\langle\nabla f(x),y-x\rangle+g(y)+\frac{1}{\epsilon}D(x,y)\},\quad\forall x\in\RR^n.
\end{equation}
{\bf Coordinate Bregman Proximal Mapping}\\
Coordinate Bregman proximal mapping $\hat{T}_{i,D,\epsilon}(x)$ is defined by
\begin{eqnarray}\label{defi:Tki}
&&\mbox{For any given index $i\in\{1,2,...,N\}$}\nonumber\\
&&\hat{T}_{i,D,\epsilon}(x)=\arg\min_{y\in\RR^n}\langle\nabla_i f(x),(y-x)_i\rangle+g_i(y_i)+\frac{1}{\epsilon}D(x,y),\quad\forall x\in\RR^n,
\end{eqnarray}
which can be viewed as the optimizer of optimization problem (AP$^{i(k)}$), where $i$ is replaced by random variable $i(k)$ and $x$ is replaced by $x^k$. In another words, if index $i$ is random variable, then $\hat{T}_{i,D,\epsilon}(x)$ is a random output.\\
\begin{lemma} Let a Bregman function $D$ and parameter $\epsilon$ be given. Let index $i$ is chosen from $\{1,2,...,N\}$ with equal probability. Suppose that Assumptions~\ref{assump1} and~\ref{assump2} hold, then for any $x\in\RR^n$, the following assertions are true.
{\rm
\begin{itemize}
\item[(i)] $\hat{T}_{i,D,\epsilon}(x)$ (see~\eqref{defi:Tki}) and $T_{D,\epsilon}(x)$ (see~\eqref{defi:Tk}) are single value.
\item[(ii)] For any given $\RR^n\rightarrow\RR$ function $\psi$, we have
\begin{equation}\label{eq:observation-0} \mathbb{E}_i\psi\big{(}\hat{T}_{i,D,\epsilon}(x)\big{)}=\frac{1}{N}\sum_{i=1}^{N}\psi\big{(}\hat{T}_{i,D,\epsilon}(x)\big{)}. \end{equation}
\item[(iii)] For $j=1,...,N$, for any mapping $\varphi$, we have
\begin{equation}\label{eq:observation-1} \mathbb{E}_i\varphi\bigg{(}\big{(}\hat{T}_{i,D,\epsilon}(x)\big{)}_j\bigg{)}=\frac{1}{N}\varphi\bigg{(}\big{(}T_{D,\epsilon}(x)\big{)}_j\bigg{)}+\left(1-\frac{1}{N}\right)\varphi(x_j). \end{equation}
Specifically, it guarantees that
\begin{equation}\label{eq:observation-2}
\mathbb{E}_i\hat{T}_{i,D,\epsilon}(x)=\frac{1}{N}T_{D,\epsilon}(x)+\left(1-\frac{1}{N}\right)x.
\end{equation}
\end{itemize}
}
\end{lemma}
\begin{proof}
\begin{itemize}
\item[(i)] Since $g$ is semi-convex with modules $\rho$ in Assumption~\ref{assump1} and $\overline{\epsilon}<\min\{\frac{m}{L},\frac{m}{\rho}\}$ in Assumption~\ref{assump2}, by Proposition 2.4 of Zhu and Deng~\cite{ZhuDeng2019}, we have statement (i) of this lemma.
\item[(ii)] \& (iii) Trivially.
\end{itemize}
\end{proof}
The above mappings and functions enjoy favorable properties, which are summarized in the following propositions.
%%%%%%%%BASIC PROPERTIES OF BREGMAN FUNCTION and EXISTENCE RESULTS%%%%%%%%%%%%%%%%%%%%%%%%%%%%%%%%%%%%%%%%
\begin{proposition}{\bf(Global properties of Bregman type mappings and functions)}\label{prop:Ek} Let a Bregman function $D$ and parameter $\epsilon$ be given. Let index $i$ is chosen from $\{1,2,...,N\}$ with equal probability. Suppose that Assumptions~\ref{assump1} and \ref{assump2} hold, then for any $x\in\RR^n$,
{\rm
\begin{itemize}
\item[(i)] $F\left(\hat{T}_{i,D,\epsilon}(x)\right)-F(x)\leq-a\|x-\hat{T}_{i,D,\epsilon}(x)\|^2$ and $a=\frac{m-\overline{\epsilon}L}{2\overline{\epsilon}}$;
\item[(ii)] $N\mathbb{E}_{i}F\big{(}\hat{T}_{i,D,\epsilon}(x)\big{)}-(N-1)F(x)\leq E_{D,\epsilon}(x)-\frac{N}{2}(\frac{m}{\overline{\epsilon}}-L)\mathbb{E}_{i}\|x-\hat{T}_{i,D,\epsilon}(x)\|^2$.
\end{itemize}
}
\end{proposition}
\begin{proof}
\begin{itemize}
\item[(i)] Since $\hat{T}_{i,D,\epsilon}(x)$ is the optimizer of minimization problem in~\eqref{defi:Tki}, we have that
\begin{eqnarray}
\langle\nabla_i f(x),(\hat{T}_{i,D,\epsilon}(x)-y)_i\rangle+g_i\left((\hat{T}_{i,D,\epsilon}(x))_i\right)-g_i(y_i)\nonumber\\+\frac{1}{\epsilon}D(x,\hat{T}_{i,D,\epsilon}(x))-\frac{1}{\epsilon}D(x,y)\leq0.
\end{eqnarray}
Take $y=x$ in above inequality, by the fact $\langle\nabla_i f(x),(\hat{T}_{i,D,\epsilon}(x)-x)_i\rangle=\langle\nabla f(x),\hat{T}_{i,D,\epsilon}(x)-x\rangle$ and $g_i\left((\hat{T}_{i,D,\epsilon}(x))_i\right)-g_i(x_i)=g\left(\hat{T}_{i,D,\epsilon}(x)\right)-g(x)$ we have
\begin{eqnarray}
\langle\nabla f(x),\hat{T}_{i,D,\epsilon}(x)-x\rangle+g\left(\hat{T}_{i,D,\epsilon}(x)\right)-g(x)+\frac{1}{\epsilon}D(x,\hat{T}_{i,D,\epsilon}(x))\leq0.
\end{eqnarray}
By the gradient Lipschitz of $f$ and $\frac{1}{\epsilon}D(x,\hat{T}_{i,D,\epsilon}(x))\geq\frac{m}{2\overline{\epsilon}}\|x-\hat{T}_{i,D,\epsilon}(x)\|^2$, we have that
\begin{eqnarray}
F\left(\hat{T}_{i,D,\epsilon}(x)\right)-F(x)&\leq&-\frac{m-\overline{\epsilon}L}{2\overline{\epsilon}}\|x-\hat{T}_{i,D,\epsilon}(x)\|^2\nonumber\\
&=&-a\|x-\hat{T}_{i,D,\epsilon}(x)\|^2,
\end{eqnarray}
with $a=\frac{m-\overline{\epsilon}L}{2\overline{\epsilon}}$.
\item[(ii)] By the definition of BP Envelope Function $E_{D,\epsilon}(x)$ we have that
\begin{eqnarray}\label{eq:Ek-0}
E_{D,\epsilon}(x)&=&f(x)+\langle\nabla f(x),T_{D,\epsilon}(x)-x\rangle+g\big{(}T_{D,\epsilon}(x)\big{)}+\frac{1}{\epsilon}D\big{(}x,T_{D,\epsilon}(x)\big{)}\nonumber\\
&\geq&f(x)+\langle\nabla f(x),T_{D,\epsilon}(x)-x\rangle+g\big{(}T_{D,\epsilon}(x)\big{)}+\frac{m}{2\overline{\epsilon}}\|x-T_{D,\epsilon}(x)\|^2\nonumber\\
\end{eqnarray}
By~\eqref{eq:observation-2}, we have
\begin{eqnarray}\label{eq:Ek-1}
\langle\nabla f(x),T_{D,\epsilon}(x)-x\rangle&=&\langle\nabla f(x),N\mathbb{E}_{i}\hat{T}_{i,D,\epsilon}(x)-(N-1)x-x\rangle\nonumber\\
&=&\langle\nabla f(x),N\mathbb{E}_{i}\hat{T}_{i,D,\epsilon}(x)-Nx\rangle\nonumber\\
&=&N\mathbb{E}_{i}\langle\nabla f(x),\hat{T}_{i,D,\epsilon}(x)-x\rangle.
\end{eqnarray}
For any $j=1,...,N$, by~\eqref{eq:observation-1} with $\varphi(\cdot)=g_j(\cdot)$, we have
\begin{eqnarray*}
g_j\big{(}(T_{D,\epsilon}(x))_j\big{)}=N\mathbb{E}_{i}g_j\big{(}(\hat{T}_{i,D,\epsilon}(x))_j\big{)}-(N-1)g_j(x_j).
\end{eqnarray*}
It follows that
\begin{eqnarray}\label{eq:Ek-2}
g\big{(}T_{D,\epsilon}(x)\big{)}=N\mathbb{E}_{i}g\big{(}\hat{T}_{i,D,\epsilon}(x)\big{)}-(N-1)g(x).
\end{eqnarray}
For any $j=1,...,N$, again using~\eqref{eq:observation-1} with $\varphi(y)=\|x-y\|^2$, we have
\begin{eqnarray*}
\|(x-T_{D,\epsilon}(x))_j\|^2=N\mathbb{E}_{i}\|(x-\hat{T}_{i,D,\epsilon}(x))_j\|^2.
\end{eqnarray*}
Therefore
\begin{eqnarray}\label{eq:Ek-3}
\|x-T_{D,\epsilon}(x)\|^2=N\mathbb{E}_{i}\|x-\hat{T}_{i,D,\epsilon}(x)\|^2.
\end{eqnarray}
Together~\eqref{eq:Ek-0},~\eqref{eq:Ek-1},~\eqref{eq:Ek-2} and~\eqref{eq:Ek-3}, we have
\begin{eqnarray}
E_{D,\epsilon}(x)&\geq&f(x)+N\mathbb{E}_{i}\langle\nabla f(x),\hat{T}_{i,D,\epsilon}(x)-x\rangle+N\mathbb{E}_{i}g\big{(}\hat{T}_{i,D,\epsilon}(x)\big{)}-(N-1)g(x)\nonumber\\
&&+\frac{Nm}{2\overline{\epsilon}}\mathbb{E}_{i}\|x-\hat{T}_{i,D,\epsilon}(x)\|^2\nonumber\\
&\geq&N\mathbb{E}_{i}F\big{(}\hat{T}_{i,D,\epsilon}(x)\big{)}-(N-1)F(x)+\frac{N(m-\overline{\epsilon} L)}{2\overline{\epsilon}}\mathbb{E}_{i}\|x-\hat{T}_{i,D,\epsilon}(x)\|^2\nonumber\\
&&\qquad\qquad\qquad\qquad\qquad\qquad\qquad\qquad\mbox{(since $\nabla f$ is $L$-Lipschitz)}
\end{eqnarray}
\end{itemize}
\end{proof}
The following proposition provides an upper bound for the function value under the level-set subdifferential error bound condition.
\begin{proposition}{\bf(Uniform estimate of value proximity by stochastic coordinate Bregman proximal mappings)}\label{prop:2}
Let Bregman function $D$ and parameter $\epsilon$ be given. Let index $i$ is chosen from $\{1,2,...,N\}$ with equal probability. Suppose that Assumptions~\ref{assump1} and \ref{assump2} hold. Moreover, assume the level-set subdifferential error bound condition holds at $\overline{x}$ for positive numbers $c_0$, $\eta$ and $\nu$. If $x\in\mathfrak{B}(\overline{x};\frac{\eta}{2},\frac{\nu}{\mathcal{N}})$ with $\mathcal{N}\geq\frac{2\overline{\epsilon}\nu}{m-\overline{\epsilon}L}/(\frac{\eta}{2})^2$, then there exist positive number $\theta_1=1+c_0(L+\frac{M}{\underline{\epsilon}})$, $\theta_2=\frac{3}{2}L+\frac{M}{2\underline{\epsilon}}$, $\kappa=(\theta_1)^2\theta_2$ and $b=\frac{2\overline{\epsilon}N^2\kappa}{m-\overline{\epsilon}L}+N$
\begin{itemize}
\item[{\rm(i)}] $dist(x,[F\leq\overline{F}])\leq\theta_1\|T_{D,\epsilon}(x)-x\|$;
\item[{\rm(ii)}] $N\mathbb{E}_{i}F\big{(}\hat{T}_{i,D,\epsilon}(x)\big{)}-(N-1)F(x)-\overline{F}\leq E_{D,\epsilon}(x)-\overline{F}\leq\theta_2{dist}^2(x,[F\leq\overline{F}])$;
\item[{\rm(iii)}] $N\mathbb{E}_{i}F\big{(}\hat{T}_{i,D,\epsilon}(x)\big{)}-(N-1)F(x)-\overline{F}\leq N^2\kappa\mathbb{E}_{i}\|\hat{T}_{i,D,\epsilon}(x)-x\|^2$;
\item[{\rm(iv)}] $F(x)-\overline{F}\leq b\left[F(x)-\mathbb{E}_{i}F\big{(}\hat{T}_{i,D,\epsilon}(x)\big{)}\right]$;
\item[{\rm(v)}] $\mathbb{E}_{i}F\left(\hat{T}_{i,D,\epsilon}(x)\right)-\overline{F}\leq \frac{b-1}{b}\left[F(x)-\overline{F}\right]$.
\end{itemize}
\end{proposition}
\begin{proof}
\begin{itemize}
\item[{\rm(i)}] This statement from Theorem 3.1 in Zhu and Deng~\cite{ZhuDeng2019}.
\item[{\rm(ii)}] The first inequality of this statement is followed by statement (ii) of Proposition~\ref{prop:Ek} in this paper. The second inequality is derived by Proposition 3.1 in Zhu and Deng~\cite{ZhuDeng2019}.
\item[{\rm(iii)}] Together statement (i) and (ii), we have
\begin{eqnarray}
&&N\mathbb{E}_{i}F(\hat{T}_{i,D,\epsilon}(x))-(N-1)F(x)-\overline{F}\nonumber\\
&\leq&(\theta_1)^2\theta_2\|T_{D,\epsilon}(x)-x\|^2\nonumber\\
&=&\kappa\|T_{D,\epsilon}(x)-x\|^2\nonumber\\
&=&\kappa\|N\mathbb{E}_{i}\hat{T}_{i,D,\epsilon}(x)-(N-1)x-x\|^2\quad\mbox{(by~\eqref{eq:observation-2})}\nonumber\\
&=&\kappa\|N\mathbb{E}_{i}\hat{T}_{i,D,\epsilon}(x)-Nx\|^2\nonumber\\
&\leq&N^2\kappa\mathbb{E}_{i}\|\hat{T}_{i,D,\epsilon}(x)-x\|^2.\quad\mbox{(by the convexity of $\|\cdot\|^2$)}
\end{eqnarray}
\item[{\rm(iv)}] From statement (iii) of this Proposition, we have that
\begin{eqnarray}
F(x)-F(\overline{x})&\leq& N^2\kappa\mathbb{E}_{i}\|\hat{T}_{i,D,\epsilon}(x)-x\|^2+N\mathbb{E}_{i}\left[F(x)-F(\hat{T}_{i,D,\epsilon}(x))\right]\nonumber\\
&\leq&\left(\frac{2\overline{\epsilon}N^2\kappa}{m-\overline{\epsilon}L}+N\right)\mathbb{E}_{i}\left[F(x)-F(\hat{T}_{i,D,\epsilon}(x))\right]\nonumber\\
&&\qquad\qquad\qquad\qquad\qquad\qquad\qquad\qquad\mbox{(by (i) of Proposition 2.1)}\nonumber\\
&=&b\mathbb{E}_{i}\left[F(x)-F(\hat{T}_{i,D,\epsilon}(x))\right]
\end{eqnarray}
where $b=\frac{2\overline{\epsilon}N^2\kappa}{m-\overline{\epsilon}L}+N$.
\item[{\rm(v)}] This statement is directly from statement (iv) of this proposition.
\end{itemize}
\end{proof}
Moreover, we introduce the following {\bf Property A} which will be used in the convergence rate analysis.\\
{\bf Property (A)} Let $x, \overline{x}\in\RR^n$ be given, We say $\hat{T}_{i,D,\epsilon}(x)$ satisfies {\bf Property A} if we have $F\big{(}\hat{T}_{i,D,\epsilon}(x)\big{)}\geq F(\overline{x})=\overline{F}$, $\forall i\in\{1,...,N\}$.\\
\begin{lemma}\label{lemma:ppa}
Let Bregman function $D$ and parameter $\epsilon$ be given. Suppose Assumptions~\ref{assump1} and~\ref{assump2} hold. Let $x\in\mathfrak{B}(\overline{x};\frac{\eta}{2},\frac{\nu}{\mathcal{N}})$ with $\mathcal{N}\geq\frac{2\overline{\epsilon}\nu}{m-\overline{\epsilon}L}/(\frac{\eta}{2})^2$ be given. If $\hat{T}_{i,D,\epsilon}(x)$ satisfies Property A, then for all $i\in\{1,...,N\}$, we have that $\|x-\hat{T}_{i,D,\epsilon}(x)\|\leq\frac{\eta}{2}$ and $\hat{T}_{i,D,\epsilon}(x)\in\mathfrak{B}(\overline{x};\eta,\frac{\nu}{\mathcal{N}})$.
\end{lemma}
\begin{proof}
Since $\hat{T}_{i,D,\epsilon}(x)$ satisfy {\bf Property A}, $F\big{(}\hat{T}_{i,D,\epsilon}(x)\big{)}\geq\overline{F}$, $\forall i=1,..,N$. By statement (i) of Proposition~\ref{prop:Ek} and $x\in\mathfrak{B}(\overline{x};\frac{\eta}{2},\frac{\nu}{\mathcal{N}})$, we have that
\begin{equation}
a\|x-\hat{T}_{i,D,\epsilon}(x)\|^2\leq F(x)-F\big{(}\hat{T}_{i,D,\epsilon}(x)\big{)}\leq F(x)-\overline{F}\leq\frac{\nu}{\mathcal{N}},\quad\forall i=1,...,N,
\end{equation}
with $a=\frac{m-\overline{\epsilon}L}{2\overline{\epsilon}}$. Since $\mathcal{N}\geq\frac{2\overline{\epsilon}\nu}{m-\overline{\epsilon}L}/(\frac{\eta}{2})^2$, consequently,
\begin{equation}
\|x-\hat{T}_{i,D,\epsilon}(x)\|\leq\frac{\eta}{2},\quad\forall i=1,...,N.
\end{equation}
Since $\|x-\overline{x}\|\leq\frac{\eta}{2}$, it follows that $\|\hat{T}_{i,D,\epsilon}(x)-\overline{x}\|\leq\|x-\hat{T}_{i,D,\epsilon}(x)\|+\|x-\overline{x}\|\leq\eta$, $\forall i=1,...,N$. It follows that $\hat{T}_{i,D,\epsilon}(x)\in\mathfrak{B}(\overline{x};\eta,\frac{\nu}{\mathcal{N}})$.
\end{proof}
%%%%%%%%%%%%%%%%%%%%%%%%%%%%%
%%%%%%%%%%SECTION convergence%%%%%%%%%%
%%%%%%%%%%%%%%%%%%%%%%%%%%%%%
\section{Convergence analysis for VBSCD}\label{sec:convergence}
In this section, we discuss the convergence behavior of the sequences generated by the VBSCD method. In section~\ref{sec:convergence} and~\ref{sec:convergence rate}, we assume that the variable Bregman functions $D^k$ and parameters $\epsilon^k$ uniformly satisfy Assumption~\ref{assump2}. In algorithm VBSCD, the indices $i(k)$, $k=0,1,2,\ldots$ are random variables. After $k$ iterations, the VBSCD method generates a random output $x^{k+1}$. We denote by $\xi_k$ is a filtration generated by the random variable $i(0),i(1),\ldots,i(k)$, i.e.,
$$\xi_{k}\overset{def}{=}\{i(0),i(1),\ldots,i(k)\}, \xi_{k}\subset\xi_{k+1}.$$
Additionally, we define that $\xi=(\xi_{k})_{k\in\mathbb{N}}$,  $\mathbb{E}_{\xi_{k+1}}=\mathbb{E}(\cdot|\xi_{k})$ is the condition expectation w.r.t. $\xi_{k}$ and the condition expectation in term of $i(k)$ given $i(0),i(1),\ldots,i(k-1)$ as $\mathbb{E}_{i(k)}$. Several basic properties of sequences $\{x^k\}$ and $\{F(x^k)\}$ are summarized in the following proposition.
\begin{proposition}\label{prop:convergence} Suppose that Assumptions~\ref{assump1} and \ref{assump2} hold. Let $\{x^k\}$ be a sequence generated by the VBSCD method. Then the following assertions hold:
\begin{itemize}
\item[{\rm(i)}] $F(x^k)-F(x^{k+1})\geq a\|x^k-x^{k+1}\|^2$, $\forall k\in\mathbb{N}$ and $a=\frac{m-\overline{\epsilon}L}{2\overline{\epsilon}}$;
\item[{\rm(ii)}] $\lim\limits_{k\rightarrow\infty}F(x^{k})=F_{\zeta}$ $a.s.$ with $F_{\zeta}\geq F^*$ is some random variable, $\lim\limits_{k\rightarrow\infty}\|x^k-x^{k+1}\|=0$ $a.s.$ and $\lim\limits_{k\rightarrow\infty}\|x^k-T_{D^k,\epsilon^k}(x^k)\|=0$ $a.s.$;
\item[{\rm(iii)}] The random variable sequence $\{x^k\}$ generated by VBSCD is almost surely bounded;
\item[{\rm(iv)}] Any cluster point of a realization sequence generated by VBSCD is a critical point of $F$.
\end{itemize}
\end{proposition}
\begin{proof}
\begin{itemize}
\item[(i)] The claim follows directly from (i) of Proposition~\ref{prop:Ek}.
\item[(ii)] Take expectation of $i(k)$ on both side of statement (i) of this proposition, we have
\begin{eqnarray}\label{eq:27}
&&\mathbb{E}_{i(k)}[F(x^{k+1})-F^*]\nonumber\\
&\leq& [F(x^{k})-F^*]-\mathbb{E}_{i(k)}a\|x^k-x^{k+1}\|^2\nonumber\\
&\leq& [F(x^{k})-F^*]-a\|x^k-\mathbb{E}_{i(k)}x^{k+1}\|^2\qquad\qquad\mbox{(by convexity of $\|\cdot\|^2$)}\nonumber\\
&=&[F(x^{k})-F^*]-\frac{a}{N}\|x^k-T_{D^k,\epsilon^k}(x^k)\|^2.\qquad\quad\mbox{(by~\eqref{eq:observation-2})}
\end{eqnarray}
By the Robbins-Siegmund's Lemma~\cite{RS1985}, we have $\lim\limits_{k\rightarrow\infty}F(x^{k})=F_{\zeta}$ $a.s.$ with $F_{\zeta}\geq F^*$ is some random variable and $\sum\limits_{k=0}^{\infty}\|x^k-T_{D^k,\epsilon^k}(x^k)\|^2<+\infty$ $a.s.$. Further, due to the almost sure convergence of sequence $\{F(x^k)\}$, it can easily get that $\lim\limits_{k\rightarrow\infty}[F(x^k)-F(x^{k+1})]=0$ $a.s.$. Together with statement (i) of this proposition we have $\lim\limits_{k\rightarrow\infty}\|x^k-x^{k+1}\|=0$ $a.s.$.\\
Moreover, $\sum\limits_{k=0}^{\infty}\|x^k-T_{D^k,\epsilon^k}(x^k)\|^2<+\infty$ $a.s.$ implies that $\lim\limits_{k\rightarrow\infty}\|x^k-T_{D^k,\epsilon^k}(x^k)\|=0$ $a.s.$.
\item[(iii)] From statement (ii) $\lim\limits_{k\rightarrow\infty}F(x^{k})=F_{\zeta}$ $a.s.$ with $F_{\zeta}\geq F^*$ is some random variable, then the almost surely boundness of $\{x^k\}$ comes from Assumption 1, $F=(f+g)$ is level bounded.
\item[(iv)] By statement (ii) of Proposition 2.3-2.5 in Zhu and Deng 2019~\cite{ZhuDeng2019} and $\lim\limits_{k\rightarrow\infty}\|x^k-T_{D^k,\epsilon^k}(x^k)\|=0$ $a.s.$ in statement (ii) of this proposition, we have that any cluster point of a realization sequence generated by VBSCD is a critical point of $F$.
\end{itemize}
\end{proof}
%%%%%%%%%%%%%%%%%
%%%%%%%%%LinearConvergence%%%%
%%%%%%%%%%%%%%%%%%%%%%
\section{Linear Convergence of VBSCD}\label{sec:convergence rate}
This section will provide the linear convergence of VBSCD under the level-set subdifferential error bound condition. First we propose a lemma which will be used in the convergence rate analysis.
\begin{lemma}\label{lemma:RS}
Suppose Assumptions~\ref{assump1} and~\ref{assump2} hold. Let $\{x^k\}$ be a sequence generated by VBSCD, if $\mathbb{E}_{\xi_{k-1}}\frac{1}{2\sqrt{b}}\sum\limits_{l=0}^k\sqrt{F(x^l)-\frac{1}{N}\sum\limits_{i(l)=1}^NF\left(\hat{T}_{i(l),D^l,\epsilon^l}(x^{l})\right)}\leq\left[F(x^0)-\overline{F}\right]^{\frac{1}{2}}$, then there exists positive number $d$, which is independent of $k$, such that $\sum\limits_{l=0}^{k}\sqrt{F(x^l)-\frac{1}{N}\sum\limits_{i(l)=1}^NF\left(\hat{T}_{i(l),D^l,\epsilon^l}(x^{l})\right)}\leq d<+\infty$, $a.s.$.
\end{lemma}
\begin{proof} This results directly by the basic property of expectation.
\end{proof}
Under the level-set subdifferential error bound condition, next proposition will show that the sequence of random variable $\{x^k\}$ generated by the VBSCD method almost surely belong to $\mathfrak{B}(\overline{x},\frac{\eta}{2},\frac{\nu}{\mathcal{N}})$.
\begin{proposition}{\bf(Almost surely finite length property of sequence $\{x^k\}$)}\label{prop:3.1}
Suppose Assumptions~\ref{assump1} and~\ref{assump2} hold. Furthermore, we assume that the level-set subdifferential error bound holds at the point $\overline{x}$ with $\eta>0$ and $\nu>0$. Let $a$, $b$, $\mathcal{N}$ and $d$ be constants given in Proposition~\ref{prop:Ek},~\ref{prop:2} and Lemma~\ref{lemma:RS} respectively. Suppose that $x^0$ satisfies the following conditions:
\begin{eqnarray}
\mbox{\rm{(1)}}&& \overline{F}\leq F(x^0)<\overline{F}+\frac{\nu}{\mathcal{N}};\label{eq:condition1}\\
\mbox{\rm{(2)}}&&\|x^0-\overline{x}\|+\frac{d\sqrt{b}}{2\sqrt{a}}+\frac{1}{\sqrt{a}}\left[F(x^0)-\overline{F}\right]^{\frac{1}{2}}<\frac{\eta}{2}.\label{eq:condition2}
\end{eqnarray}
Assume moreover that
\begin{eqnarray}
\mbox{\rm{(3)}}\quad\hat{T}_{i,D,\epsilon}(x^k)\quad\mbox{satisfy\quad{\bf Property A}},\quad\forall k\in\mathbb{N}.\label{eq:condition3}
\end{eqnarray}
Then the following statements hold.
\begin{itemize}
\item[{\rm(i)}] $x^k\in\mathfrak{B}(\overline{x},\frac{\eta}{2},\frac{\nu}{\mathcal{N}})$, $\forall k\in\mathbb{N}$ $a.s.$;
\item[{\rm(ii)}] $\sum\limits_{k=0}^{+\infty}\|x^k-x^{k+1}\|<+\infty$ $a.s.$ (finite length property), and the sequence $\{x^k\}$ converges to a random variable $x$;
\item[{\rm(iii)}] $\sum\limits_{k=0}^{+\infty}\|\mathbb{E}_{\xi_{k-1}}x^k-\mathbb{E}_{\xi_k}x^{k+1}\|<+\infty$, and the sequence $\{\mathbb{E}_{\xi_{k-1}}x^k\}$ converges to a point $\mathbb{E}_{\xi}x$.
\end{itemize}
\end{proposition}
\begin{proof}
\begin{itemize}
\item[{\rm(i)}] From Assumptions, obviously, $x^0\in\mathfrak{B}(\overline{x};\frac{\eta}{2},\frac{\nu}{\mathcal{N}})$. By (i) of Proposition~\ref{prop:convergence}, we have $F(x^1)\leq F(x^0)\leq\overline{F}+\frac{\nu}{\mathcal{N}}$. By using~\eqref{eq:condition3} and Lemma~\ref{lemma:ppa} we have $x^1\in\mathfrak{B}(\overline{x};\eta,\frac{\nu}{\mathcal{N}})$ and $F(x^1)\geq\overline{F}$. Moreover, $\|x^1-x^0\|\leq\sqrt{\frac{F(x^0)-F(x^1)}{a}}\leq\sqrt{\frac{F(x^0)-\overline{F}}{a}}$. Combining the triangle inequality, we have $\|x^1-\overline{x}\|\leq\|x^1-x^0\|+\|x^0-\overline{x}\|\leq\sqrt{\frac{F(x^0)-\overline{F}}{a}}+\|x^0-\overline{x}\|$. By~\eqref{eq:condition2}, it follows that $x^1\in\mathfrak{B}(\overline{x};\frac{\eta}{2},\frac{\nu}{\mathcal{N}})$.\\
    Now suppose $x^l\in\mathfrak{B}(\overline{x};\frac{\eta}{2},\frac{\nu}{\mathcal{N}})$ for $l=0,...,k$ and $x^k\neq x^{k+1}$. Again using~\eqref{eq:condition3} and Lemma~\ref{lemma:ppa}, we have $x^{k+1}\in\mathfrak{B}(\overline{x};\eta,\frac{\nu}{\mathcal{N}})$ and $F(x^0)\geq F(x^1)\geq F(x^2)\geq\cdots\geq F(x^k)\geq F(x^{k+1})\geq\overline{F}$.\\
    We need to show that $x^{k+1}\in\mathfrak{B}(\overline{x};\frac{\eta}{2},\frac{\nu}{\mathcal{N}})$ $a.s.$. By the concavity of function $h(y)=y^{\frac{1}{2}}$, we have
\begin{eqnarray}\label{eq:concavity-1}
\left[F(x^l)-\overline{F}\right]^{\frac{1}{2}}-\left[F(x^{l+1})-\overline{F}\right]^{\frac{1}{2}}\geq\frac{1}{2}\frac{F(x^l)-F(x^{l+1})}{\left[F(x^l)-\overline{F}\right]^{\frac{1}{2}}}.
\end{eqnarray}
Combing statement (iv) of Proposition~\ref{prop:2}, above inequality follows that
\begin{eqnarray}\label{eq:concavity-2}
\left[F(x^l)-\overline{F}\right]^{\frac{1}{2}}-\left[F(x^{l+1})-\overline{F}\right]^{\frac{1}{2}}\geq\frac{1}{2\sqrt{b}}\frac{F(x^l)-F(x^{l+1})}{\sqrt{\mathbb{E}_{i(l)}\left[F(x^l)-F(x^{l+1})\right]}}.
\end{eqnarray}
Take expectation with respect to $i(l)$ for~\eqref{eq:concavity-2}, it follows
\begin{eqnarray}\label{eq:concavity-3}
\left[F(x^l)-\overline{F}\right]^{\frac{1}{2}}-\mathbb{E}_{i(l)}\left[F(x^{l+1})-\overline{F}\right]^{\frac{1}{2}}&\geq&\frac{1}{2\sqrt{b}}\frac{\mathbb{E}_{i(l)}\left[F(x^l)-F(x^{l+1})\right]}{\sqrt{\mathbb{E}_{i(l)}\left[F(x^l)-F(x^{l+1})\right]}}\nonumber\\
&=&\frac{1}{2\sqrt{b}}\sqrt{\mathbb{E}_{i(l)}\left[F(x^l)-F(x^{l+1})\right]},\nonumber\\
\end{eqnarray}
or
\begin{eqnarray}\label{eq:concavity-4}
&&\mathbb{E}_{i(l)}\left[F(x^{l+1})-\overline{F}\right]^{\frac{1}{2}}\nonumber\\
&\leq&\left[F(x^l)-\overline{F}\right]^{\frac{1}{2}}-\frac{1}{2\sqrt{b}}\sqrt{\mathbb{E}_{i(l)}\left[F(x^l)-F(x^{l+1})\right]}\nonumber\\
&=&\left[F(x^l)-\overline{F}\right]^{\frac{1}{2}}-\frac{1}{2\sqrt{b}}\bigg{[}F(x^l)-\frac{1}{N}\sum_{i(l)=1}^{N}F\big{(}T_{i(l),D,\epsilon}(x^{l})\big{)}\bigg{]}^{\frac{1}{2}}.
\end{eqnarray}
Taking expectation with respect to $\xi_{l-1}$,~\eqref{eq:concavity-4} follows that
\begin{eqnarray}\label{eq:concavity-4-0}
&&\mathbb{E}_{\xi_l}\left[F(x^{l+1})-\overline{F}\right]^{\frac{1}{2}}+\mathbb{E}_{\xi_{l-1}}\sum_{\tau=0}^{l}\frac{1}{2\sqrt{b}}\bigg{[}F(x^{\tau})-\frac{1}{N}\sum_{i(\tau)=1}^{N}F\big{(}T_{i(\tau),D,\epsilon}(x^{\tau})\big{)}\bigg{]}^{\frac{1}{2}}\nonumber\\
&\leq&\mathbb{E}_{\xi_{l-1}}\left[F(x^l)-\overline{F}\right]^{\frac{1}{2}}+\mathbb{E}_{\xi_{l-2}}\sum_{\tau=0}^{l-1}\frac{1}{2\sqrt{b}}\bigg{[}F(x^{\tau})-\frac{1}{N}\sum_{i(\tau)=1}^{N}F\big{(}T_{i(\tau),D,\epsilon}(x^{\tau})\big{)}\bigg{]}^{\frac{1}{2}}\nonumber\\
&\leq&\left[F(x^0)-\overline{F}\right]^{\frac{1}{2}}.
\end{eqnarray}
Since $\mathbb{E}_{\xi_l}\left[F(x^{l+1})-\overline{F}\right]^{\frac{1}{2}}\geq0$, it follows that
\begin{eqnarray}
\mathbb{E}_{\xi_{l-1}}\sum_{\tau=0}^{l}\frac{1}{2\sqrt{b}}\bigg{[}F(x^{\tau})-\frac{1}{N}\sum_{i(\tau)=1}^{N}F\big{(}T_{i(\tau),D,\epsilon}(x^{\tau})\big{)}\bigg{]}^{\frac{1}{2}}\leq\left[F(x^0)-\overline{F}\right]^{\frac{1}{2}}.
\end{eqnarray}
By Lemma~\ref{lemma:RS}, we have
\begin{eqnarray}\label{eq:concavity-4-1}
\sum_{\tau=0}^{l}\sqrt{\mathbb{E}_{i(\tau)}\left[F(x^{\tau})-F(x^{\tau+1})\right]}&=&\sum_{\tau=1}^{l}\bigg{[}F(x^{\tau})-\frac{1}{N}\sum_{i(\tau)=1}^{N}F\big{(}T_{i(\tau),D,\epsilon}(x^{\tau})\big{)}\bigg{]}^{\frac{1}{2}}\nonumber\\
&\leq& d,\quad a.s.,
\end{eqnarray}
with $l=0,...,k$.\\
Again combining~\eqref{eq:concavity-2} and (i) in Proposition~\ref{prop:convergence}, we have
\begin{eqnarray}\label{eq:concavity-5}
\left[F(x^l)-\overline{F}\right]^{\frac{1}{2}}-\left[F(x^{l+1})-\overline{F}\right]^{\frac{1}{2}}&\geq&\frac{1}{2\sqrt{b}}\frac{F(x^l)-F(x^{l+1})}{\sqrt{\mathbb{E}_{i(l)}\left[F(x^l)-F(x^{l+1})\right]}}\nonumber\\
&\geq&\frac{a}{2\sqrt{b}}\frac{\|x^l-x^{l+1}\|^2}{\sqrt{\mathbb{E}_{i(l)}\left[F(x^l)-F(x^{l+1})\right]}},
\end{eqnarray}
or
\begin{eqnarray}\label{eq:concavity-6}
&&\|x^l-x^{l+1}\|^2\nonumber\\
&\leq&\frac{2\sqrt{b}}{a}\sqrt{\mathbb{E}_{i(l)}\left[F(x^l)-F(x^{l+1})\right]}\bigg{\{}\left[F(x^l)-\overline{F}\right]^{\frac{1}{2}}-\left[F(x^{l+1})-\overline{F}\right]^{\frac{1}{2}}\bigg{\}}\nonumber\\
\end{eqnarray}
It follows from $2\sqrt{d_1d_2}\leq d_1+d_2$ with nonnegative $d_1$ and $d_2$ that
\begin{eqnarray}\label{eq:concavity-7}
&&2\|x^l-x^{l+1}\|\nonumber\\
&\leq&\frac{\sqrt{b}}{\sqrt{a}}\sqrt{\mathbb{E}_{i(l)}\left[F(x^l)-F(x^{l+1})\right]}+\frac{2}{\sqrt{a}}\bigg{\{}\left[F(x^l)-\overline{F}\right]^{\frac{1}{2}}-\left[F(x^{l+1})-\overline{F}\right]^{\frac{1}{2}}\bigg{\}}\nonumber\\
\end{eqnarray}
Summing~\eqref{eq:concavity-7} for $l=0,...,k$, we obtain
\begin{eqnarray}\label{eq:concavity-8}
\sum_{l=0}^{k}\|x^{l}-x^{l+1}\|\leq\sum\limits_{l=0}^{k}\frac{\sqrt{b}}{2\sqrt{a}}\sqrt{\mathbb{E}_{i(l)}\left[F(x^l)-F(x^{l+1})\right]}+\frac{1}{\sqrt{a}}\left[F(x^0)-\overline{F}\right]^{\frac{1}{2}}.
\end{eqnarray}
By~\eqref{eq:concavity-4-1}, we have
\begin{eqnarray}\label{eq:concavity8-1}
\sum\limits_{l=0}^{k}\|x^{l}-x^{l+1}\|\leq\frac{d\sqrt{b}}{2\sqrt{a}}+\frac{1}{\sqrt{a}}\left[F(x^0)-\overline{F}\right]^{\frac{1}{2}}<+\infty,\quad a.s..
\end{eqnarray}
Combining triangle inequality,~\eqref{eq:concavity8-1} and~\eqref{eq:condition2}, we have
\begin{eqnarray}\label{eq:concavity-9}
\|x^{k+1}-\overline{x}\|&\leq&\|x^{k+1}-x^{k}\|+\|x^{k}-\overline{x}\|\nonumber\\
&\leq&\sum_{l=0}^{k}\|x^{l}-x^{l+1}\|+\|x^{0}-\overline{x}\|\nonumber\\
&\leq&\frac{d\sqrt{b}}{2\sqrt{a}}+\frac{1}{\sqrt{a}}\left[F(x^0)-\overline{F}\right]^{\frac{1}{2}}+\|x^{0}-\overline{x}\|\nonumber\\
&\leq&\frac{\eta}{2}, \quad a.s..
\end{eqnarray}
It follows that $x^{k+1}\in\mathfrak{B}(\overline{x};\frac{\eta}{2},\frac{\nu}{\mathcal{N}})$, $a.s.$.\\
Therefore, $x^k\in\mathfrak{B}(\overline{x};\frac{\eta}{2},\frac{\nu}{\mathcal{N}})$, $\forall k\in\mathbb{N}$ $a.s.$.
\item[{\rm(ii)}] A direct consequence of~\eqref{eq:concavity8-1} is, for all $k$,
\begin{eqnarray}\label{eq:concavity-9-1}
\sum_{k=0}^{+\infty}\|x^k-x^{k+1}\|<+\infty\quad\mbox{a.s.}.
\end{eqnarray}
~\eqref{eq:concavity-9-1} implies that the sequence $\{x^k\}$ converges to a random variable $x$.
\item[{\rm(iii)}] Take expectation respect to $\xi_k$ on both side of statement (ii) of this Proposition, we have that
\begin{eqnarray}\label{eq:concavity-9-2}
\sum_{k=0}^{+\infty}\mathbb{E}_{\xi_k}\|x^k-x^{k+1}\|<+\infty.
\end{eqnarray}
By the convexity of $\|\cdot\|$ and the fact $\mathbb{E}_{\xi_k}x^k=\mathbb{E}_{\xi_{k-1}}x^k$, we have that
\begin{eqnarray}\label{eq:concavity-9-3}
\sum_{k=0}^{+\infty}\|\mathbb{E}_{\xi_{k-1}}x^k-\mathbb{E}_{\xi_k}x^{k+1}\|&=&\sum_{k=0}^{+\infty}\|\mathbb{E}_{\xi_{k}}x^k-\mathbb{E}_{\xi_k}x^{k+1}\|\qquad\mbox{(since $\mathbb{E}_{\xi_k}x^k=\mathbb{E}_{\xi_{k-1}}x^k$)}\nonumber\\
&\leq&\sum_{k=0}^{+\infty}\mathbb{E}_{\xi_{k}}\|x^k-x^{k+1}\|\quad\mbox{(by the convexity of $\|\cdot\|$)}\nonumber\\
&<&+\infty.\qquad\mbox{(by~\eqref{eq:concavity-9-2})}
\end{eqnarray}
Moreover,~\eqref{eq:concavity-9-3} implies that the sequence $\{\mathbb{E}_{\xi_{k-1}}x^k\}$ converges to the point $\mathbb{E}_{\xi}x$.
\end{itemize}
\end{proof}
To further study the linear convergence of VBSCD, we need the following assumption.
\begin{assumption}\label{assump4}
The set $\{x|F(x)\leq\overline{F}\}\cap\mathbb{B}(\overline{x};\eta)$ is convex.
\end{assumption}
Here we note that, local convex and local quasi-convex function $F$ satisfy Assumption~\ref{assump4}. In robust statistics, there are many popular functions which are both quasi-convex and semi-convex, such as SCAD, MCP, etc. (see~\cite{SCAD, quasiconvex, MCP})
\subsection{Linear convergence of VBSCD under LS-EB at critical point}
For given $x^0$ and a realization sequence $\{x^k(\omega)\}$ generated by VBSCD, let $\overline{x}$ be a cluster point of $\{x^k(\omega)\}$. Therefore, there is $k_0>0$ such that $x^{k_0}(\omega)$ satisfy~\eqref{eq:condition1} and~\eqref{eq:condition2}. Let
$\xi_k^{\omega}$ be a filtration generated by the random variable $i(k_0)$,...,$i(k)$, i.e.,
$$\xi_{k}^{\omega}\overset{def}{=}\{i_{\omega}(0),i_{\omega}(1),\ldots,i_{\omega}(k_0-1),i(k_0),\ldots,i(k)\},\quad\xi_{k}^{\omega}\subset\xi_{k+1}^{\omega},$$ where $i_{\omega}(j)$, $j\in\langle0,(k_0-1)\rangle$ is fixed index corresponding to realization $\omega$. And $\xi^{\omega}=(\xi_{k}^{\omega})_{k\in\mathbb{N}}$. By statement (iii) of Proposition~\ref{prop:convergence}, there exists positive number $r(x^0)$, for all $k\in\mathbb{N}$, $\|x^k\|\leq r(x^0)$ $a.s.$. Then we have the linear convergence as follows theorem.
\begin{theorem}[The linear convergence under LS-EB at critical point]\label{theo1} Suppose Assumption~\ref{assump1} and~\ref{assump2} hold. Let Assumption~\ref{assumpH} holds with $\rho>2r(x^0)$. Moreover, we assume that the level-set subdifferential error bound holds at the point $\overline{x}$ with $\eta>0$ and $\nu>0$. Let $b$ be constants given in Proposition~\ref{prop:2}. Considering the sequence of realization $\{x^k(\omega)\}$. Then there exist $k_0>0$, for all $k\geq k_0$ the following assertions are true.
\begin{itemize}
\item[{\rm(i)}] $\{F(x^k)\}$ converges to value $\overline{F}$ at the $Q$-linear rate of convergence by expectation; that is, for $k>k_0$, there are some $\beta=\frac{b-1}{b}\in(0,1)$ such that
\begin{equation}\label{linearrate}
\mathbb{E}_{\xi_{k}^{\omega}}[F(x^{k+1})-\overline{F}]\leq\beta\mathbb{E}_{\xi_{k-1}^{\omega}}[F(x^{k})-\overline{F}].
\end{equation}
As a consequence,
\begin{equation}\label{linearrateB}
\sum_{k=k_0+1}^{\infty}\mathbb{E}_{\xi_{k-1}^{\omega}}[F(x^{k})-\overline{F}]<+\infty.
\end{equation}
\item[{\rm(ii)}] The sequence $\{\mathbb{E}_{\xi_{k-1}^{\omega}}x^k\}$, converges to point $\mathbb{E}_{\xi^{\omega}}x\in\mathbb{B}(\overline{x};\frac{\eta}{2})$ at the $R$-linear rate of convergence; under Assumption~\ref{assump4} we have $F(\mathbb{E}_{\xi^{\omega}}x)=\overline{F}$.
\end{itemize}
\end{theorem}
\begin{proof}
Let $\Omega$ be the set of accumulation points for the realization of $\{x^k\}$. Since Assumption~\ref{assumpH} holds with $\rho>2r(x^0)$, we have that $F(x)=\overline{F}$, $\forall x\in\Omega$. Moreover, $F(x^k)\geq\overline{F}$, $\forall k\in\mathbb{N}$. Then condition~\eqref{eq:condition3} holds.\\
Together with $x^{k_0}(\omega)$ satisfy~\eqref{eq:condition1} and~\eqref{eq:condition2}, by Proposition~\ref{prop:3.1}, for all $k> k_0$, $x^{k}\in\mathfrak{B}(\overline{x};\frac{\eta}{2},\frac{\nu}{\mathcal{N}})$ $a.s.$.
\begin{itemize}
\item[(i)] For all $k> k_0$, by (v) of Proposition~\ref{prop:2} it follows that
\begin{eqnarray}\label{Q-linear}
\mathbb{E}_{i(k)}F(x^{k+1})-\overline{F}=\beta\left(F(x^k)-\overline{F}\right),
\end{eqnarray}
where $\beta=\frac{b-1}{b}\in(0,1)$. Again using the fact $x^k\in\mathfrak{B}(\overline{x};\frac{\eta}{2},\frac{\nu}{\mathcal{N}})$ $a.s.$, we take the expectation with respect to $\xi_{k-1}^{\omega}$ for inequality~\eqref{Q-linear}, we obtain that
\begin{eqnarray}\label{Q-linear-2}
\mathbb{E}_{\xi_{k}^{\omega}}[F(x^{k+1})-\overline{F}]=\beta\mathbb{E}_{\xi_{k-1}^{\omega}}[F(x^k)-\overline{F}].
\end{eqnarray}
\item[(ii)] We now derive the R-linear rate of convergence of $\{\mathbb{E}_{\xi_{k-1}^{\omega}}x^k\}$. Taking expectation with respect to $\xi_{k}^{\omega}$ for (i) in Proposition~\ref{prop:convergence}, we have
\begin{eqnarray}
\mathbb{E}_{\xi_{k-1}^{\omega}}F(x^k)-\mathbb{E}_{\xi_{k}^{\omega}}F(x^{k+1})\geq a\mathbb{E}_{\xi_{k}^{\omega}}\|x^k-x^{k+1}\|^2.
\end{eqnarray}
Thus
\begin{eqnarray*}
\mathbb{E}_{\xi_{k}^{\omega}}\|x^k-x^{k+1}\|^2&\leq&\frac{1}{a}\bigg{[}\mathbb{E}_{\xi_{k-1}^{\omega}}\big{(}F(x^k)-\overline{F}\big{)}-\mathbb{E}_{\xi_{k}^{\omega}}\big{(}F(x^{k+1})-\overline{F}\big{)}\bigg{]}\\
                 &\leq&\frac{1}{a}\mathbb{E}_{\xi_{k-1}^{\omega}}[F(x^k)-\overline{F}]\\
                 &\leq&\frac{\beta^{k-k_0}}{a}[F\big{(}x^{k_0}(\omega)\big{)}-\overline{F}]\qquad\qquad\qquad\mbox{ (by~\eqref{Q-linear-2})}.
\end{eqnarray*}
From the convexity of $\|\cdot\|^2$ and above inequality, we see that
$$\|\mathbb{E}_{\xi_{k}^{\omega}}(x^k-x^{k+1})\|\leq\sqrt{\mathbb{E}_{\xi_{k}^{\omega}}\|x^k-x^{k+1}\|^2}\leq\hat{M}(\sqrt{\beta})^{k-k_0},$$
where $\hat{M}=\sqrt{\frac{F(x^{k_0}(\omega))-\overline{F}}{a}}$.
By statement (iii) of Proposition~\ref{prop:3.1}, we have $\sum\limits_{k=k_0}^{+\infty}\|\mathbb{E}_{\xi_{k-1}^{\omega}}x^k-\mathbb{E}_{\xi_{k}^{\omega}}x^{k+1}\|<+\infty$ and $\mathbb{E}_{\xi_{k-1}^{\omega}}x^k$ converges to the point $\mathbb{E}_{\xi^{\omega}}x$. Hence,
$$\|\mathbb{E}_{\xi_{k-1}^{\omega}}x^k-\mathbb{E}_{\xi^{\omega}}x\|\leq\sum_{l=k}^{\infty}\|\mathbb{E}_{\xi_{l-1}^{\omega}}x^l-\mathbb{E}_{\xi_{l}^{\omega}}x^{l+1}\|\leq\frac{\hat{M}}{1-\sqrt{\beta}}(\sqrt{\beta})^{k-k_0}.$$ This
shows that $\{\mathbb{E}_{\xi_{k-1}^{\omega}}x^k\}$ converges to $\mathbb{E}_{\xi^{\omega}}x$ at the R-linear rate; that is, $$\limsup_{k\rightarrow\infty}\sqrt[k]{\|\mathbb{E}_{\xi_{k-1}^{\omega}}x^k-\mathbb{E}_{\xi^{\omega}}x\|}=\sqrt{\beta}<1.$$
Since $x^k\in\mathfrak{B}(\overline{x};\frac{\eta}{2},\frac{\nu}{\mathcal{N}})$ $a.s.$, then $x\in\mathbb{B}(\overline{x};\frac{\eta}{2})$ and $\mathbb{E}_{\xi^{\omega}}x\in\mathbb{B}(\overline{x};\frac{\eta}{2})$. Let $\mathbb{W}$ be the set realization of $x$. For $F\left(x(\omega')\right)=\overline{F}$, $\omega'\in\mathbb{W}$, under Assumption~\ref{assump4}, we have $F\left(\mathbb{E}_{\xi^{\omega}}x\right)\leq\overline{F}$. Since $x^k\rightarrow x$ $a.s.$ and $x^k\in\mathfrak{B}(\overline{x};\frac{\eta}{2},\frac{\nu}{\mathcal{N}})$, we have $F\left(\mathbb{E}_{\xi^{\omega}}x\right)\geq\overline{F}$. Then it follows that $F\left(\mathbb{E}_{\xi^{\omega}}x\right)=\overline{F}$.
\end{itemize}
\end{proof}
\subsection{Linear convergence to a local minima}\label{local minima}
Thorough this subsection, let $\overline{x}$ be a local minimum on $\mathbb{B}(\overline{x};\eta)$. If the level-set subdifferential error bound holds at point $\overline{x}$ with $\eta>0$ and $\nu>0$, we will show that, under Assumption~\ref{assump3} and special selection of the initial point, sequence $\{x^k\}$ almost surely belongs to $\mathfrak{B}(\overline{x},\frac{\eta}{2},\frac{\nu}{\mathcal{N}})$ and the linear convergence to local minima of VBSCD
\begin{theorem}[The linear convergence to a local minima]\label{theo2} Suppose Assumption~\ref{assump1} and~\ref{assump2} hold. Moreover, we assume that the level-set subdifferential error bound holds at the point $\overline{x}$ with $\eta>0$ and $\nu>0$. Let $a$, $b$ be constant given in Proposition~\ref{prop:2}, Assumption~\ref{assump3} holds with $\sigma=\frac{\eta}{2}$, $\rho=\eta$, $\delta=\frac{\nu}{\mathcal{N}}$ and $\mathfrak{a}\leq a$. Let $x^0$ satisfy~\eqref{eq:condition1} and~\eqref{eq:condition2} and the sequence $\{x^k\}$ be generated by the VBSCD method. Then the following assertions are true.
\begin{itemize}
\item[{\rm(i)}] $\{F(x^k)\}$ converges to value $\overline{F}$ at the $Q$-linear rate of convergence by expectation; that is, there are some $\beta=\frac{b-1}{b}\in(0,1)$ such that
\begin{equation}\label{linearrate}
\mathbb{E}_{\xi_{k-1}}F(x^{k})-\overline{F}\leq\beta^k\bigg{(}F\big{(}x^{0}\big{)}-\overline{F}\bigg{)}.
\end{equation}
As a consequence,
\begin{equation}\label{linearrateB}
\sum_{k=0}^{\infty}\mathbb{E}_{\xi_{k-1}}[F(x^{k})-\overline{F}]<+\infty.
\end{equation}
\item[{\rm(ii)}] The sequence $\{\mathbb{E}_{\xi_{k-1}}x^k\}$, converges to point $\mathbb{E}_{\xi}x\in\mathbb{B}(\overline{x};\frac{\eta}{2})$ at the $R$-linear rate of convergence; under Assumption~\ref{assump4}, we have $F(\mathbb{E}_{\xi}x)=\overline{F}$ and $\mathbb{E}_{\xi}x$ is a local minimum.
\end{itemize}
\end{theorem}
\begin{proof}
The proof of this theorem is similar to the proof of Theorem~\ref{theo1}. The difference of proof is the following:
\begin{itemize}
\item[(1)] Together Assumption 4 (Growth condition) holds with $\sigma=\frac{\eta}{2}$, $\rho=\eta$, $\delta=\frac{\nu}{\mathcal{N}}$ and $\mathfrak{a}\leq a$ and statement (i) in Proposition 3.1, we have that $\forall k\in\mathbb{N}$, $x^k\in\mathbb{B}(\overline{x},\frac{\eta}{2})$ implies $x^{k+1}\in\mathbb{B}(\overline{x},\eta)$. Since $\overline{x}$ is local minimum on $\mathbb{B}(\overline{x},\eta)$, then $F(x^{k+1})\geq\overline{F}$ and condition~\eqref{eq:condition3} holds. Since $x^0$ satisfy~\eqref{eq:condition1} and~\eqref{eq:condition2}, by Proposition~\ref{prop:3.1}, we have that $\forall k\in\mathbb{N}$, $x^{k}\in\mathfrak{B}(\overline{x};\frac{\eta}{2},\frac{\nu}{\mathcal{N}})$ $a.s.$.
\item[(2)] Under Assumption~\ref{assump4}, we have $F\left(\mathbb{E}_{\xi}x\right)=\overline{F}$ and $\mathbb{E}_{\xi}x\in\mathbb{B}(\overline{x};\eta)$. Since $\overline{x}$ is a local minimum on $\mathbb{B}(\overline{x};\eta)$, we have that $\mathbb{E}_{\xi}x$ is also a local minimum on $\mathbb{B}(\overline{x};\eta)$.
\end{itemize}
\end{proof}
%%%%%%%%%%%%
%%%%%%%%%%KER PROP FOR CONV RATE
%%%%%%%%%%%%%%%%%%%%%%%
\subsection{Linear convergence to a global minima}
Thorough this subsection, let $\overline{x}$ be a global minimum. If the level-set subdifferential error bound holds at point $\overline{x}$ with $\eta>0$ and $\nu>0$, we will show the linear convergence of VBSCD.
\begin{theorem}[The linear convergence to a global minima]\label{theo3} Suppose Assumption~\ref{assump1} and~\ref{assump2} hold. Moreover, we assume that the level-set subdifferential error bound holds at the point $\overline{x}$ with $\eta>0$ and $\nu>0$. Let $b$ and $\mathcal{N}$ be constant given in Proposition~\ref{prop:2}. There exist $\sigma\in(0,\frac{\eta}{2})$ such that the inequalities
\begin{equation}
\|x^0-\overline{x}\|<\sigma,\quad\overline{F}<F(x^0)<\overline{F}+\frac{\nu}{\mathcal{N}}
\end{equation}
implies that any realization of the sequence $\{x^k\}$ generated by VBSCD satisfies
\begin{itemize}
\item[{\rm(i)}] $x^k\in\mathfrak{B}(\overline{x};\frac{\eta}{2},\frac{\nu}{\mathcal{N}})$,
\item[{\rm(ii)}] $\{F(x^k)\}$ converges to value $\overline{F}$ at the $Q$-linear rate of convergence by expectation; that is, there are some $\beta=\frac{b-1}{b}\in(0,1)$ such that
\begin{equation}
\mathbb{E}_{\xi_{k-1}}F(x^{k})-\overline{F}\leq\beta^k\bigg{(}F\big{(}x^{0}\big{)}-\overline{F}\bigg{)}.
\end{equation}
As a consequence,
\begin{equation}
\sum_{k=0}^{\infty}\mathbb{E}_{\xi_{k-1}}[F(x^{k})-\overline{F}]<+\infty.
\end{equation}
\item[{\rm(iii)}] The sequence $\{\mathbb{E}_{\xi_{k-1}}x^k\}$, converges to point $\mathbb{E}_{\xi}x\in\mathbb{B}(\overline{x};\frac{\eta}{2})$ at the $R$-linear rate of convergence; under Assumption~\ref{assump4}, we have $F(\mathbb{E}_{\xi}x)=\overline{F}$ and $\mathbb{E}_{\xi}x$ is a global minimum.
\end{itemize}
\end{theorem}
\begin{proof}
It is a straightforward variant of Theorem~\ref{theo1} and~\ref{theo2}.
\end{proof}
%%%%%%%%%%%%%%%%%%%
%\begin{figure}
%\begin{center}
%\centering
%\includegraphics[width=4.2in]{Fig.png}
%\caption{Figure}
%\label{Fig1}
%\end{center}
%\end{figure}
%%%%%%%%%%%%%%%%%%%%
%%%%%%%%%%%%GAP FUNCTION CONVERGENCE RATE%%%%%%%%%%
%%%%%%%%%%%%%%%%%%%%%%%%%%%%%%%%%%%%%%%
%%%%1-12-2019%%%%%%%%%%%%%%%%%%%%%%%%%%%%
%%%%%%%%%%%%%%%%%%%
%%%%%%%%%%%%%%%%%%%
%%%%%REFERENCES%%%%%%%%%
%%%%%%%%%%%%%%%%%%%%%

\end{document}